\documentclass[11pt]{article}
\pdfoutput=1
\usepackage[T1]{fontenc}
\usepackage[utf8x]{inputenc}
\usepackage{graphicx}
\usepackage{caption}
\usepackage{subcaption}
\usepackage{amsmath}
\usepackage{amssymb}
\usepackage{epsfig}
\usepackage{changebar}
\usepackage{color}
\usepackage{latexsym,amsfonts,amscd,amsthm}
\usepackage{fancyhdr}
\usepackage{chemarr}
\usepackage{tikz}
\usepackage{cite}
\usepackage{enumerate}

\numberwithin{equation}{section}

\newcounter{abc}

\newcounter{roem}

\newtheorem{theorem}{Theorem}[section]
\newtheorem{lemma}[theorem]{Lemma}
\newtheorem{proposition}[theorem]{Proposition}

\theoremstyle{definition}

\newtheorem{remark}{Remark}

\def\R{\mathbb{R}}

\def\D{\mathcal D}

\def\V{\mathcal V}

\def\O{\mathcal O}

\DeclareMathOperator{\rank}{rank}
\DeclareMathOperator{\im}{im}

\newcommand{\abs}[1]{\ensuremath{\left\vert#1\right\vert}}
\newcommand{\norm}[2][\relax]{\ifx#1\relax \ensuremath{\left\Vert#2\right\Vert} \else \ensuremath{\left\Vert#2\right\Vert_{#1}}\fi}

\newcommand{\footnoteremember}[2]{%
  \footnote{#2}
  \newcounter{#1}
  \setcounter{#1}{\value{footnote}}
}
\newcommand{\footnoterecall}[1]{%
  \footnotemark[\value{#1}]
}

\begin{document}

\title{Quasi-steady state reduction for the Michaelis-Menten reaction-diffusion system}

\author{Martin Frank\footnote{MathCCES, RWTH Aachen, 52056 Aachen, Germany}, Christian Lax\footnoteremember{1}{Lehrstuhl A f. Mathematik, RWTH Aachen, 52056 Aachen, Germany}, Sebastian Walcher\footnoterecall{1}\footnote{Corresponding author: walcher@matha.rwth-aachen.de}, Olaf Wittich\footnoterecall{1}}

\maketitle 

\begin{abstract}
 The Michaelis-Menten mechanism is probably the best known model for an enzyme-catalyzed reaction. For spatially homogeneous concentrations, QSS reductions are well known, but this is not the case when chemical species are allowed to diffuse. We will discuss QSS reductions for both the irreversible and reversible Michaelis-Menten reaction in the latter case, given small initial enzyme concentration and slow diffusion. Our work is based on a heuristic method to obtain an ordinary differential equation which admits reduction by Tikhonov-Fenichel theory. We will not give convergence proofs but we provide numerical results that support the accuracy of the reductions. 
\end{abstract}

\section{Introduction}

The Michaelis-Menten mechanism \cite{michaelismenten} is probably the best known model for an enzyme-catalyzed reaction. In this reaction network, a substrate $S$ and an enzyme $E$ combine to form a complex $C$, which degrades back to substrate and enzyme, or to product P and enzyme. In the reversible setting there is also a back reaction combining $E$ and $P$ to complex. The reaction scheme thus reads
   \[
    E+S \xrightleftharpoons[k_{-1}]{k_{1}} C \xrightleftharpoons[k_{-2}]{k_{2}} E+P.
  \]  
In the irreversible case one assumes that product and enzyme cannot combine to form complex, i.e. one has $C\xrightarrow{k_{2}}E+P$.  Typically, no complex or product are assumed present initially. 
Assuming mass action kinetics and spatially homogeneous concentrations, the evolution of the concentrations $(s,e,c,p)$ of $S,E,C,P$ can be described by a system of four ordinary differential equations, from which by stoichiometry one obtains  a two-dimensional system (first discussed from a  mathematical perspective  by Briggs and Haldane \cite{briggshaldane}). Employing the familiar quasi-steady state (QSS) assumption for complex, based on  small initial concentration of enzyme, further reduces the system to dimension one.\\ 
For reaction systems, quasi-steady state (QSS) assumptions frequently lead to singular perturbation problems for which the classical theories of Tikhonov \cite{tikh} and Fenichel \cite{fenichel} are applicable. (Moreover, one should note results by Hoppensteadt \cite{Hoppensteadt} on unbounded time intervals; see also \cite{lws}.)  \\
For spatially inhomogeneous concentrations in a reaction vessel, thus for reaction-diffusion systems, Tikhonov's and Fenichel's theory is not applicable since their fundamental results are limited to finite dimensional systems. Therefore, reaction-diffusion systems are much more difficult to analyze, and only partial results are known. As for the Michaelis-Menten reaction with diffusion and small initial enzyme concentration, Britton \cite{britton} and Yannacopoulos et al. \cite{Yannacopoulos} derived QSS reductions with the additional assumptions of immobile complex and enzyme. Kalachev et al. \cite{kkkpz} used asymptotic expansions with respect to a small parameter to obtain results about the behavior of the solutions under different time scales for diffusion, with the diffusion time scale different from the time scale for the slow reaction part. (As \cite{kkkpz} indicates, even finding candidates for reduced reaction-diffusion systems may be a nontrivial task.)  Starting from different assumptions about the reaction mechanism (viz., smallness of certain rate constants), Bothe and Pierre \cite{BothePierre1} as well as Bisi et al. \cite{bisi} discussed reductions for a related system, including convergence proofs. \\
In the present paper we will discuss QSS reductions for both the irreversible and the reversible Michaelis-Menten reaction with diffusion, under the conditions of small initial enzyme concentration and slow diffusion. Our work is based on a heuristic method described in \cite{laxgoeke}, which utilizes a spatial discretization to obtain an ordinary differential equation which admits reduction by Tikhonov-Fenichel theory. In many relevant cases, the reduced ODE system can, in turn,  be identified as the spatial discretization of another partial differential equation system. This resulting PDE is a candidate for the reduced system and, as pointed out in \cite{laxgoeke}, it is the only possible candidate. In the present paper we will not discuss convergence issues, which seem to be quite technically  involved, but we provide numerical simulations that support the accuracy of the reduction. \\
The plan of the paper is as follows. In Section 2 we will briefly recall the most important aspects of the spatially homogeneous system  and moreover note some general features of  the inhomogeneous case. \\
In Section 3 the ``classical'' QSS assumption is discussed, i.e. we assume small initial enzyme concentration and slow diffusion. We first review some relevant results from the literature, and give an informal description of the reduction procedure from \cite{laxgoeke}. Following a (degenerate) scaling similar to the one in Heineken, Tsuchiya and Aris \cite{hta} we derive a reduction via the approach in \cite{laxgoeke}; to the authors' knowledge, the form of the reduced PDE system has not been known in the literature to date. \\ The reduction is consistent with the spatially homogeneous case, thus setting the diffusion constants equal to zero yields the usual Michaelis-Menten equation. The degenerate scaling seems unavoidable in the PDE case (while one can circumvent it for the ODE), thus we need to go beyond the classical singular perturbation reduction due to Tikhonov and Fenichel. The scaling requires a consistency condition which is intuitively likely to hold in general; we can justify it mathematically in the case when enzyme and complex diffuse at the same rate. In Section 4 we present numerical simulations which exhibit very good agreement with the reduced system. \\
In the Appendix, employing the heuristic method from \cite{laxgoeke}, we carry out the necessary computations for the reductions and also determine suitable initial values for the reduced system.

\section{Preliminaries}

\subsection{The spatially homogeneous setting}
We recall some facts about the Michaelis-Menten reaction with homogeneously distributed concentrations. The evolution of the concentrations $(s,e,c,p)$ of $S,E,C,P$ is governed by the four-dimensional ordinary differential equation
  \begin{align*}
   &\dot s=-k_1es+k_{-1}c\\
   &\dot e=-k_1es+(k_{-1}+k_2)c-k_{-2}ep\\
   &\dot c=k_1es-(k_{-1}+k_2)c+k_{-2}ep\\
   &\dot p=k_2c-k_{-2}ep.
  \end{align*}
This system admits the (stoichiometric) first integrals $\Psi_1(s,e,c,p)=e+c$ and $\Psi_2(s,e,c,p)=s+c+p$. Therefore a two-dimensional system remains:
  \begin{align}
   &\dot s=-k_1e_0s+(k_1s+k_{-1})c\label{mm1}\\
   &\dot c=k_1e_0s-(k_1s+k_{-1}+k_2)c+k_{-2}(e_0-c)(s_0-s-c)\label{mm2},
  \end{align}
where $s_0$ and $e_0$ are the initial concentrations of $S$ and $E$, and initially no product $P$ or complex $C$ are present. The system is called irreversible whenever $k_{-2}=0$, and reversible otherwise. The most common quasi-steady state assumption is that the initial enzyme concentration is small, one considers $e_0=\varepsilon e_0^*$ in the asymptotic limit $\varepsilon\to 0$, for the irreversible system.\\
 Heineken, Tsuchiya and Aris \cite{hta} were the first to discuss the Michaelis-Menten system from the perspective of singular perturbations, and Segel and Slemrod \cite{ss} were the first to directly prove a rigorous convergence result for the unbounded time interval: Writing \eqref{mm1}--\eqref{mm2} in the slow time scale $\tau=\varepsilon t$
  \begin{align}
   &s'=-k_1se_0^*+\varepsilon^{-1}(k_1s+k_{-1})c\label{mm1slow}\\
   &c'=k_1se_0^*-\varepsilon^{-1}(k_1s+k_{-1}+k_2)c\label{mm2slow},
  \end{align}
the solutions of \eqref{mm1slow}--\eqref{mm2slow} converge for all $t_0>0$ uniformly on $[t_0,\infty)$ to the solutions of
  \begin{equation}\label{red}
   s'= - \frac{k_1k_2se_0^*}{k_1s+k_{-1}+k_2}
  \end{equation}
on the asymptotic slow manifold $\V=\{(s,0),\ s\geq0\}$ as $\varepsilon\to0$. (Below we will sometimes change between time scales without mentioning this explicitly.) \\
Both the approach by Heineken et al. \cite{hta} and the proof by Segel and Slemrod \cite{ss} use appropriate scalings of the variables, in particular they introduce $z:=c/e_0$. It is possible to avoid such a scaling, which becomes degenerate as $e_0\to 0$, in the spatially homogeneous case (see e.g. \cite{gw}) but as it turns out we will need to utilize such a degenerate scaling to obtain a reduction when concentrations are not homogeneously distributed in the reaction vessel.\\
We will also discuss the reversible Michaelis-Menten system, which appears less frequently in the literature; in part this may be due to the unwieldy expression for the QSS reduction; see Miller and Alberty \cite{MiAl}. The singular perturbation reduction (see \cite{nw11} and \cite{gwz2}) of \eqref{mm1slow}--\eqref{mm2slow} for $k_{-2}>0$ and $e_0=\varepsilon e_0^*$ leads to 
  \begin{equation}\label{redrev}
   s'= - \frac{(k_1k_2s+k_{-1}k_{-2}(s-s_0))e_0^*}{k_1s+k_{-2}(s_0-s)+k_{-1}+k_2}
  \end{equation}
on the asymptotic slow manifold $\V=\{(s,0),\ s\geq0\}$ as $\varepsilon\to0$. (Here, uniform convergence again holds on $[t_0,\infty)$; see \cite{lws}). Both the QSS and the singular perturbation reductions agree up to first order in the small parameter; see  \cite{gwz2}.

\subsection{The spatially inhomogeneous setting}
When the concentrations are inhomogeneously distributed and diffusion is present then the system is described by a reaction-diffusion equation. Thus, let $\Omega$ be a bounded region with a smooth boundary and let $\delta_s,\delta_e,\delta_c,\delta_p\geq0$ denote the diffusion constants. The governing equations are 
  \begin{align}
   \partial_{t} s&= \delta_s \Delta s-k_1se+k_{-1} c, &\text{in } (0,\infty)\times \Omega \label{mm1diff}\\
   \partial_{t} e&= \delta_e\Delta e -k_1se+ (k_{-1}+k_2) c-k_{-2}ep,  &\text{in } (0,\infty)\times \Omega \label{mm2diff} \\
   \partial_{t} c&= \delta_c\Delta c +k_1se-(k_{-1}+k_2) c+k_{-2}ep,  &\text{in } (0,\infty)\times \Omega \label{mm3diff} \\
   \partial_{t} p&= \delta_p\Delta p+k_2 c -k_{-2}ep,  &\text{in } (0,\infty)\times \Omega \label{mm4diff}
  \end{align}
with continuous initial values
  \[
   s(0,x)=s_0(x),\quad e(0,x)=e_0(x),\quad c(0,x)=c_0(x),\quad p(0,x)=p_0(x),\quad \text{in }  \Omega
  \]
and one has Neumann boundary conditions
  \[
   \frac{\partial s}{\partial \nu}=\frac{\partial e}{\partial \nu}=\frac{\partial c}{\partial \nu}=\frac{\partial p}{\partial \nu}=0,\quad  \text{in } (0,\infty)\times \partial\Omega
  \]
with $\frac{\partial }{\partial \nu}$ denoting the outer normal derivative. 
We collect a few general properties.

\begin{remark}\label{remark1}
\begin{itemize}
\item From Smith \cite{Smith}, Ch.~7, Thm.~3.1 and Cor.~3.2--3.3  one sees that  all the solution entries remain nonnegative for all $t>0$ whenever they are nonnegative at $t=0$. Moreover, Bothe and Rolland \cite{BotheRolland1} (see in particular Remark 1) have shown that there exists a classical solution of {class $C^{\infty}$} whenever one has initial values of class $W^{s,p}(\Omega;\R^4_+)$ for $p>1$, $s>0$.
\item When $\delta_e=\delta_c$ then 
\[
\partial_t(e+c)=\delta_e\Delta(e+c)
\]
and as a consequence of the strong maximum principle (see Smith \cite{Smith} Theorem 2.2) $e+c$ is uniformly bounded by ${\rm max}(e_0+c_0)$ for all $t\geq 0$.\\
Furthermore, in the case that $\delta_s=\delta_e=\delta_c=\delta_p$ one gets
\[
\partial_t(s+e+2c+p)=\delta_e\Delta(s+e+2c+p), 
\]
whence $s+e+2c+p$ is bounded by ${\rm max}(s_0+e_0+2c_0+p_0)$ for all $t\geq 0$; in particular nonnegativity implies that every component is bounded.   
\item The stoichiometric first integrals of the spatially homogeneous setting survive as conservation laws
\[\frac{1}{\abs{\Omega}}\int_{\Omega}e(0,x)+c(0,x)\: dx=\frac{1}{\abs{\Omega}}\int_{\Omega}e_0(x)+c_0(x)\: dx\]
resp.
\[\frac{1}{\abs{\Omega}}\int_{\Omega}s(0,x)+c(0,x)+p(0,x)\: dx=\frac{1}{\abs{\Omega}}\int_{\Omega}s_0(x)+c_0(x)+p_0(x)\: dx,\]
but a reduction of dimension (i.e.,  elimination of certain variables) is no longer possible. 
\item In the irreversible case one may consider only the first three equations as their right-hand sides do not depend on $p$.
\item Results regarding the long time behavior of solutions of the reversible Michaelis-Menten reaction can be found in Elia\v{s} \cite{elias}.
\end{itemize}
\end{remark}


\section{Reduction given slow diffusion and small initial enzyme concentration}

\subsection{Review of results in the literature}
As noted above, there exists no counterpart to Tikhonov's and Fenichel's theorems for infinite dimensional systems, hence the reduction of reaction-diffusion equations is not possible in a similarly direct manner.\\
Regarding the reduction of the Michaelis-Menten reaction with diffusion, one sometimes finds the one-dimensional equation \eqref{red} augmented by a diffusion term for substrate, with no further argument given. This ad-hoc method is problematic, since it amounts to ignoring diffusion in the reduction step.
The appropriate approach is to start with the full system \eqref{mm1diff}--\eqref{mm4diff} and consider possible reductions in the limiting case of small initial concentration for enzyme, with slow diffusion. This will be the vantage point in the present paper.\\
With regard to such an approach, the authors are aware only of three papers for the irreversible system (i.e. \eqref{mm1diff}--\eqref{mm3diff} with $k_{-2}=0$). Yannacopoulos et al. \cite{Yannacopoulos} assumed $P$ and $C$ to be immobile (i.e. $\delta_e=\delta_c=0$; see their equation (71)) and gave a second order approximation for the case of a one dimensional domain (see in particular equation (80) which in lowest order reduces to the Michaelis-Menten equation for substrate, augmented by diffusion). Britton \cite{britton}, Ch.~8 gave the first order approximation
    \begin{equation}\label{immobile}
     \partial_{\tau} s= \delta_s \Delta s- \frac{k_1k_2s{(e_0+c_0)}}{k_1s+k_{-1}+k_2}
    \end{equation}
which is in agreement with the lowest order terms given in \cite{Yannacopoulos}. He made no assumptions on diffusion constants for enzyme or complex, and instead started with system \eqref{mm1}--\eqref{mm2}, augmented by diffusive terms for $s$ and $c$. This is problematic because the elimination of $e$ via stoichiometry is no longer possible when diffusion is present. Therefore Britton's  approach is limited to the case considered by Yannacopoulos at al. \cite{Yannacopoulos}. \\
Kalachev et al. \cite{kkkpz} started from \eqref{mm1diff}--\eqref{mm3diff} and considered up to three time scales, with the  slow reaction part of order $\varepsilon$ (the total initial mass of enzyme divided by the total initial mass of substrate), a fast reaction part, and diffusion of order $\delta$, deriving asymptotic expansions for the solutions and reductions in different time regimes. They did not discuss the case that slow reaction and diffusion are in the same time scale (i.e., $\delta=\varepsilon$) which we will consider. (In \cite{kkkpz}, Remark 1.2 further work was announced for this case, but apparently this has not been published yet.)

\subsection{Informal review of the reduction heuristics}\label{heuristic}
We will employ a heuristic method to construct a candidate for a reduced system that was introduced in \cite{laxgoeke}.  In contrast to the convergence property for the ODE after discretization (which is a consequence of Tikhonov's and Fenichel's theorems) we will not prove convergence here; generally this seems a very hard task (see Section \ref{conclrem}). However, as remarked in \cite{laxgoeke}, Proposition 4.3, the reduced PDE determined by the heuristics represents the only possible reduction of the reaction-diffusion system as $\varepsilon\to 0$. \\
Briefly the heuristics can be described as follows: By spatial discretization of a reaction-diffusion system which depends on a small parameter $\varepsilon$, one obtains a system of ordinary differential equations depending on $\varepsilon$. If the ODE system admits a Tikhonov-Fenichel reduction and the reduced ODE is the spatial discretization of another partial differential equation system, then we will call the latter {\em the reduced PDE of the reaction-diffusion system}. (The conditions stated above are frequently satisfied; see e.g.\cite{Lax}.)   The following results are in part based on the second author's doctoral thesis \cite{Lax}. {Detailed computations will be presented in the Appendix.}

\subsection{The irreversible case}\label{mainresultirrev}
In order to determine the reduced PDE systems, we need some preparations.
We consider first the irreversible reaction-diffusion system \eqref{mm1diff}--\eqref{mm3diff}.
Defining total enzyme concentration $y:=e+c$, we get
  \begin{align}
   \partial_{t} s&= \delta_s \Delta s-k_1s(y-c)+k_{-1} c, &\text{in } (0,\infty)\times \Omega \label{mm1diffy}\\
   \partial_{t} c&= \delta_c\Delta c +k_1s(y-c)-(k_{-1}+k_2) c,  &\text{in } (0,\infty)\times \Omega\label{mm2diffy}\\
   \partial_{t} y&= \delta_c\Delta c + \delta_e(\Delta y -\Delta c),  &\text{in } (0,\infty)\times \Omega\label{mm3diffy}
  \end{align}
with initial values $s(0,x)=s_0(x)$, $c(0,x)=c_0(x)$, $y(0,x)=e_0(x)+c_0(x)$. Our basic assumptions are:
\begin{itemize} 
\item Diffusion is slow, and therefore we introduce the scaling 
\[
\delta_z=\varepsilon \delta_z^*\text{  for  }z=s,e,c.
\]
\item  Total enzyme concentration is small for all $t\geq 0$, and therefore we set
\[
 y=\varepsilon y^*{\text{ and } c=\varepsilon c^*},\text{  and also  }e_0=\varepsilon e_0^*,\quad c_0=\varepsilon c_0^*.
\]
\end{itemize}
 Incorporating these assumptions we have
  \begin{align*}
   \partial_{t} s&= \varepsilon\delta_s^* \Delta s+\varepsilon(k_1s+k_{-1}) c^*- \varepsilon k_1s y^*, &\text{in } (0,\infty)\times \Omega \\
   \partial_{t} c^*&= \varepsilon\delta_c^* \Delta c-(k_1s+k_{-1}+k_2) c^* + k_1s y^*,  &\text{in } (0,\infty)\times \Omega\\
   \partial_{t} y^*&= \varepsilon\delta_e^*\Delta y^*+\varepsilon(\delta_c^*-\delta_e^*) \Delta c^*,  &\text{in } (0,\infty)\times \Omega
  \end{align*}
with initial values \[s(0,x)=s_0(x),\quad c^*(0,x)=c_0^*(x),\quad y^*(0,x)=y^*_0(x):=e_0^*(x)+c_0^*(x).\] In slow time $\tau=\varepsilon t$ one now finds
  \begin{align}
   \partial_{\tau} s&= \delta_s^* \Delta s+(k_1s+k_{-1}) c^*- k_1s y^*, &\text{in } (0,\infty)\times \Omega \label{mm1diffskal}\\
   \partial_{\tau} c^*&= \delta_c^* \Delta c^*-\varepsilon^{-1}(k_1s+k_{-1}+k_2) c^* + \varepsilon^{-1} k_1s y^*,  &\text{in } (0,\infty)\times \Omega\label{mm2diffskal}\\
   \partial_{\tau} y^*&= \delta_e^*\Delta y^*+\delta \Delta c^*,  &\text{in } (0,\infty)\times \Omega\label{mm3diffskal}
  \end{align}
with the abbreviation 
\begin{equation}\label{delteq}
\delta:=\delta_c^*-\delta_e^*.
\end{equation} 
We will discuss two different cases: If the diffusion constants $\delta_e^*$ and $\delta_c^*$ are close in the sense that $\delta=\varepsilon \delta^*$, then equation \eqref{mm3diffskal} reads
  \begin{equation}
   \partial_{\tau} y^*= \delta_e^*\Delta y^*+\varepsilon\delta^* \Delta c
  \end{equation}
and the reduced system for $\varepsilon\to 0$ is again a reaction-diffusion system (with a rational reaction term). Otherwise, the reduced system becomes highly nonlinear.

\begin{remark}
The argument is based on the critical assumption that the ``degenerate'' scalings $c^*=\varepsilon^{-1} c$ and $y^*=\varepsilon^{-1} y$ hold for all $t\geq 0$; to state it more precisely, one needs a uniform bound (with respect to $\varepsilon$) for $c^*$ and $y^*$. 
In the special case $\delta_c^*=\delta_e^*$  (e.g. if the molecules of enzyme and complex are of the same size; see Keener and Sneyd \cite{ksI}, Subsection 2.2.2), Remark \ref{remark1} implies that $c^*$ and $y^*$ are uniformly bounded by $e_0^*+c_0^*$. We are not able to extend this property to the case $\delta_c^*\neq\delta_e^*$, but we will verify in the Appendix that the corresponding uniform boundedness property holds for the ODEs obtained via discretization. Furthermore, numerical results indicate that degenerate scaling poses no problem for the Michaelis-Menten system (see Section \ref{numeric}).
\end{remark}

\subsubsection{Irreversible case with $\delta_c^*-\delta_e^*=\O(\varepsilon)$}\label{mainresultirrevsub1}

In this case the reduced PDE (as defined in subsection \ref{heuristic}) is given by
  \begin{align}
   \partial_{\tau} s&= \delta_s^* \Delta s-\frac{k_1k_2y^*s}{k_1s+k_{-1}+k_2}, &\text{in } (0,T)\times \Omega\label{mm1diffred1} \\
   \partial_{\tau} y^*&= \delta_e^*\Delta y^*,  &\text{in } (0,T)\times \Omega \label{mm3diffred1}
  \end{align}
on the asymptotic slow manifold \[\V=\left\{(s,c^*,y^*)\in\R^3_+,\ c^*=\frac{k_1s y^*}{k_1s+k_{-1}+k_2}\right\}.\] 
Appropriate initial values on $\V$ are given by
  $(\tilde s_0,\tilde y_0^*)=\left(s_0,y_0^*\right)$. 
This assertion is a direct consequence of Proposition \ref{irrevdiscred} in the Appendix.\\
Total enzyme concentration in the reduced equation is subject only to diffusion, and there remains a reaction-diffusion equation for substrate, with the reaction part similar to the usual Michaelis-Menten term. It is worth looking at some special cases: When $\delta_e^*=\delta_c^*=0$, $y^*=y_0^*$ is constant in time and we have
    \begin{equation*}
     \partial_{\tau} s= \delta_s \Delta s- \frac{k_1k_2s_0y_0^*}{k_1s+k_{-1}+k_2}
    \end{equation*}
as in Yannacopoulos et al. \cite{Yannacopoulos}, Equation (80) and in Britton \cite{britton}, Ch.~8. Moreover, setting all diffusion constants to zero (and assuming $c_0^*=0$ as well as constant $y_0^*$) leads to the usual spatially homogeneous reduction as given in \eqref{red}.\\
As far as the authors know, this reduced system has not appeared in the literature so far. The numerical simulations in Section \ref{numeric} indicate convergence.

\subsubsection{Irreversible case with $\delta_c^*-\delta_e^*=\O(1)$}\label{mainresultirrevsub2}

In this case the reduction is given by
  \begin{align}
   \partial_{\tau} s&= \delta_s^* \Delta s-\frac{k_1k_2y^*s}{k_1s+k_{-1}+k_2}, &\text{in } (0,T)\times \Omega\label{mm1diffred} \\
   \partial_{\tau} y^*&= \delta_e^*\Delta y^*+\delta \Delta\left(\frac{k_1y^*s}{k_1s+k_{-1}+k_2}\right),  &\text{in } (0,T)\times \Omega \label{mm3diffred}
  \end{align}
on the asymptotic slow manifold \[\V=\left\{(s,c^*,y^*)\in\R^3_+,\ c^*=\frac{k_1s y^*}{k_1s+k_{-1}+k_2}\right\}.\] 
The appropriate initial values are as before (also following from Proposition \ref{irrevdiscred}).\\
This case may be said to correspond to the one mentioned but not treated in Kalachev et al. \cite{kkkpz}; there seems to be no discussion of this in the literature. Note that now the equations for $s$ and $y^*$ are fully coupled; this is a more complex situation than before. Again, numerical simulations (Section \ref{numeric}) are in good agreement with the reduction.

\subsection{The reversible case}\label{mainresultrev}

We will determine a reduced system for the reversible Michaelis-Menten reaction with diffusion, i.e.,
  \begin{align*}
   \partial_{\tau} s&= \delta_s^* \Delta s+(k_1s+k_{-1}) c^*- k_1s y^*, &\text{in } (0,T)\times \Omega \\
   \partial_{\tau} c^*&= \delta_c^* \Delta c^*-\varepsilon^{-1}[(k_1s+k_{-1}+k_{-2}p+k_2) c^* + (k_1s+k_{-2}p) y^*],  &\text{in } (0,T)\times \Omega\\
   \partial_{\tau} y^*&= \delta_e^*\Delta y^*+\delta \Delta c^*,  &\text{in } (0,T)\times \Omega\\
   \partial_{\tau} p&= \delta_p^* \Delta p+(k_{-2}p+k_{2}) c^*- k_{-2}py^*, &\text{in } (0,T)\times \Omega
  \end{align*}
 Here we choose the same scaling as in \eqref{mm1diffskal}--\eqref{mm3diffskal} and additionally we let $\delta_p=\varepsilon \delta_p^*$. If $\delta_c^*-\delta_e^*=\O(\varepsilon)$ then we get
  \begin{align}
   \partial_{\tau} s&= \delta_s^* \Delta s-\frac{(k_1k_2s-k_{-1}k_{-2}p)y^*}{k_1s+k_{-2}p+k_{-1}+k_2}, &\text{in } (0,T)\times \Omega\label{mm1diffred3} \\
   \partial_{\tau} y^*&= \delta_e^*\Delta y^*,  &\text{in } (0,T)\times \Omega \label{mm3diffred3}\\
   \partial_{\tau} p&= \delta_p^*\Delta p+\frac{(k_1k_2s-k_{-1}k_{-2}p)y^*}{k_1s+k_{-2}p+k_{-1}+k_2},  &\text{in } (0,T)\times \Omega \label{mm4diffred3}
  \end{align}
on the asymptotic slow manifold \[\V=\left\{(s,c^*,y^*,p)\in\R^4_+,\ c^*=\frac{(k_1s+k_{-2}p) y^*}{k_1s+k_{-2}p+k_{-1}+k_2}\right\}.\] Note that \eqref{mm3diffred3} is uncoupled from the remaining system.
In case $\delta_c^*-\delta_e^*=\O(1)$ we get
  \begin{align}
   \partial_{\tau} s&= \delta_s^* \Delta s-\frac{(k_1k_2s-k_{-1}k_{-2}p)y^*}{k_1s+k_{-2}p+k_{-1}+k_2}, &\text{in } (0,T)\times \Omega\label{mm1diffred4} \\
   \partial_{\tau} y^*&= \delta_e^*\Delta y^*+\delta\Delta\left(\frac{(k_1s+k_{-2}p)y^*}{k_1s+k_{-2}p+k_{-1}+k_2}\right),  &\text{in } (0,T)\times \Omega \label{mm3diffred4}\\
   \partial_{\tau} p&= \delta_p^*\Delta p+\frac{(k_1k_2s-k_{-1}k_{-2}p)y^*}{k_1s+k_{-2}p+k_{-1}+k_2},  &\text{in } (0,T)\times \Omega \label{mm4diffred4}
  \end{align}
on the same asymptotic slow manifold, 
but here one has a fully coupled system for $s$, $y^*$ and $p$.\\
The proofs follow from Proposition \ref{revdiscred}.
In both settings, appropriate initial values are given by
$(\tilde s_0,\tilde y_0^*,\tilde p_0)=\left(s_0,y_0^*,p_0\right)$.\\ 
Again, setting all diffusion constants to zero (and assuming $c_0^*=p_0=0$ as well as constant $y_0^*$ and $s_0$) leads to $s(\tau,x)+p(\tau,x)=s_0$ and thus to the usual reduction as given in \eqref{redrev}.

%

\section{Numerical simulations}\label{numeric}
In the following we will provide numerical results that are in good agreement with the reduction given above. 
The solutions have been obtained using MATLAB's \texttt{pdepe} function. This function solves an initial-boundary value problem for spatially one-dimensional systems of parabolic and elliptic partial differential equations in the self-adjoint form 
$$
C(x,t,u,\partial_x u)\partial_t u = x^{-m} \partial_x (x^m F(x,t,u,\partial_x u)) + S(x,t,u,\partial_x u).
$$
In our case, $m=0$ and $C$ is the identity matrix.
Furthermore, in the case of system \eqref{mm1diffskal}--\eqref{mm3diffskal}, using the unknown $u = ( s, c^*, y^*)^T$, the flux $F$ and the source $S$ become
$$
F = \begin{pmatrix} \delta_s\partial_x s  \\ \delta_c^*\partial_x c^*  \\ \delta_e^*\partial_x y^*  +\delta\partial_x c^* \end{pmatrix}, \quad
S = \begin{pmatrix} (k_1s +k_{-1})c^* -k_1s y^*  \\ \varepsilon^{-1}(k_1s +k_{-1}+k_2) c^*  + \varepsilon^{-1} k_1s  y^*   \\ 0  \end{pmatrix}.
$$
The reduced system \eqref{mm1diffred}--\eqref{mm3diffred}, using the unknown $u = ( s, y^*)^T$, the flux $F$ and the source $S$ become
$$
F = \begin{pmatrix} \delta_s\partial_x s  \\ \frac{\delta k_1y(k_{-1}+k_2)}{(k_1s+k_{-1}+k_2)^2}\partial_xs+(\frac{\delta k_1s}{k_1s+k_{-1}}+\delta_e^*)\partial_x y^*  \end{pmatrix}, \quad
S = \begin{pmatrix} \frac{k_1 k_2 s y^*}{k_1 s+k_{-1}+k_2}  \\  0  \end{pmatrix}.
$$
As boundary conditions, we use homogeneous Neumann boundary conditions, i.e., for each unknown we set the spatial derivative equal to zero at the boundary.

The \texttt{pdepe} function uses a self-adjoint finite difference semi-discretization in space, and solves the obtained system ordinary differential equations by the implicit, adaptive multistep solver \texttt{ode15s}. In all our experiments we have set the tolerances to values below the accuracy we intend to observe (absolute tolerance $10^{-14}$, relative tolerance $10^{-10}$). We have used 100 equidistant grid cells.

Figure \ref{fig:InitialCondition} shows the initial condition we have used; a step function in $s$, a smooth cosine profile for $c$, and a cosine profile with an additional Gaussian bump for $y$. We have set $\delta_s=\delta_e=k_1=k_{-1}=k_2=1$ and $\delta_c=2$ (so $\delta=1$; see case \ref{mainresultirrevsub2}).

Figure \ref{fig:sol10} shows the solutions at time $T=0.005$ for $\varepsilon=1.0$. Already, one can see that the concentration $s$ is described well by the reduced system, whereas we see a discrepancy in $y$. For $\varepsilon=0.0001$, shown in Figure \ref{fig:sol00001}, to the eye there is no difference between the solutions of the original and the reduced systems. In Figure \ref{fig:convergence} we investigate the convergence of the solution of the full system to the solution of the reduced system. The error is measured in the $L^\infty$ norm in all three solution components. As $\varepsilon\to 0$, we observe rather clean first-order convergence in double-logarithmic plot. Finally, we also set $\delta_c=1$ (so $\delta=0$; see case \ref{mainresultirrevsub1}) and measure in Figure \ref{fig:convergenceglsc} and Figure \ref{fig:convergencegly} again the error. This confirms what the theory has predicted.

\begin{figure}
\centering\includegraphics[width=0.8\linewidth]{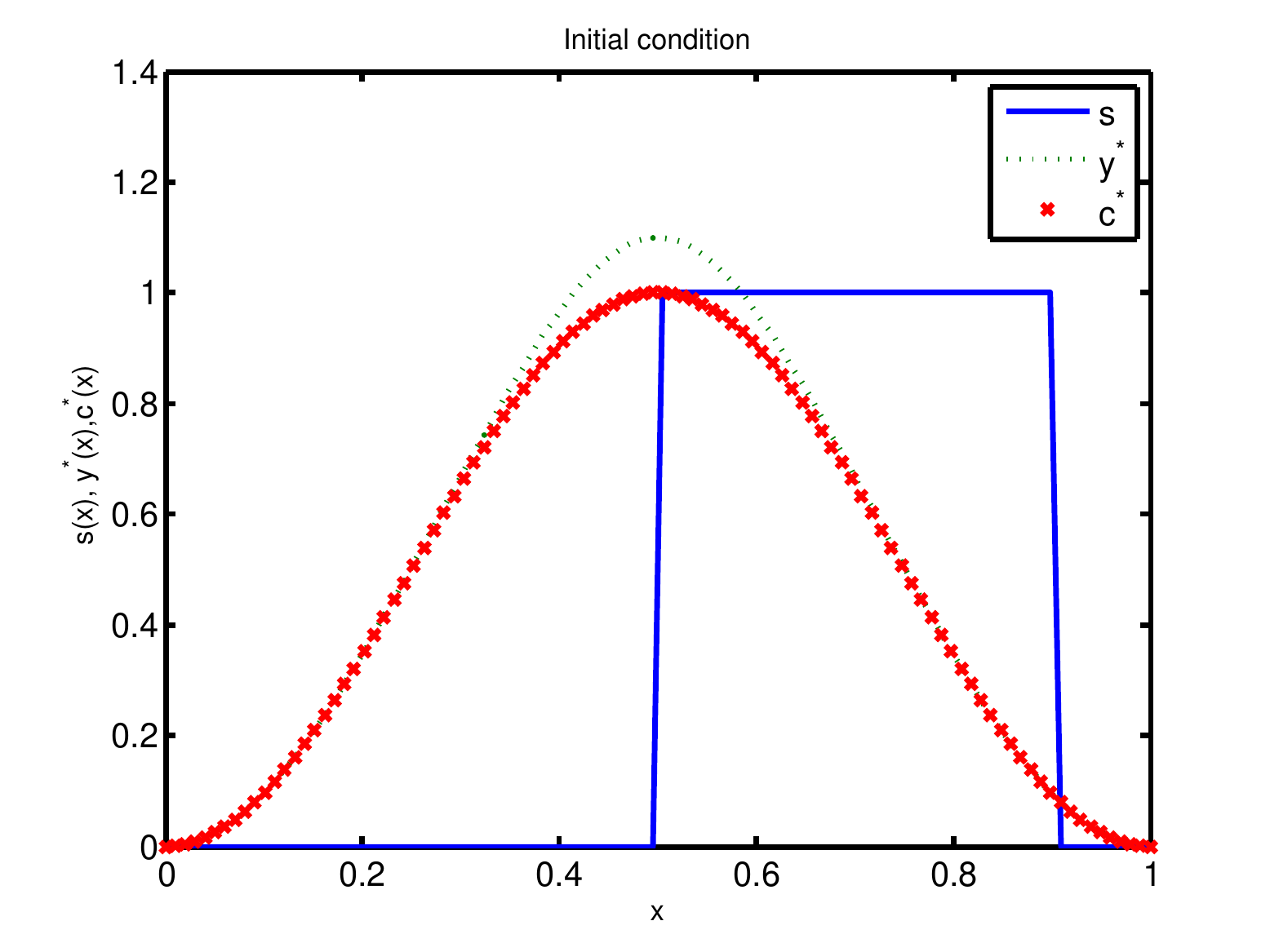}
\caption{Initial condition for $s$, $c^*$ and $y^*$.}
\label{fig:InitialCondition}
\end{figure}

\begin{figure}
\centering\includegraphics[width=0.8\linewidth]{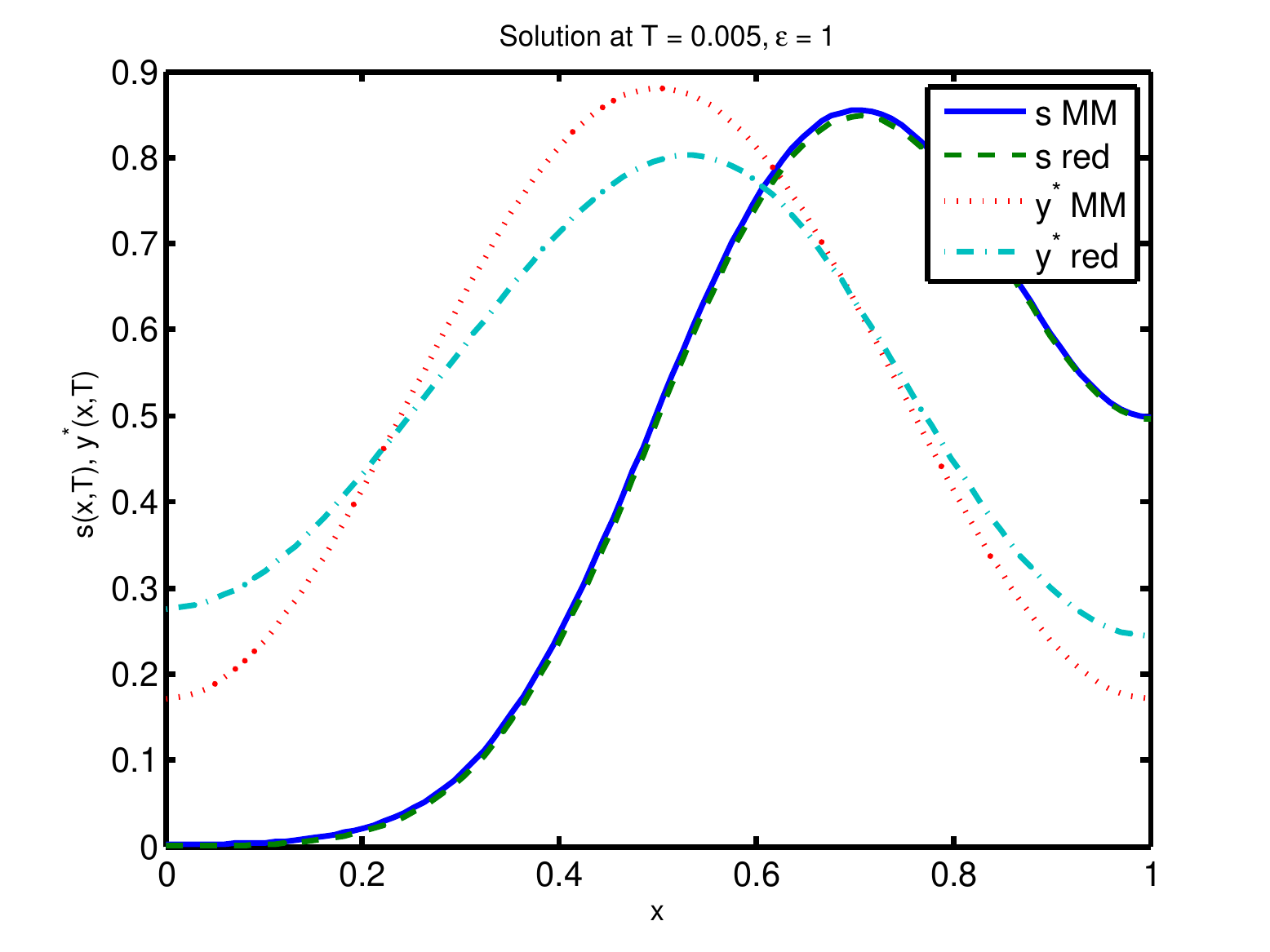}
\caption{Solutions $s$ and $y^*$ at time $T=0.005$. Comparison between Michaelis-Menten and reduced system for $\varepsilon=1.0$.}
\label{fig:sol10}
\end{figure}

\begin{figure}
\centering\includegraphics[width=0.8\linewidth]{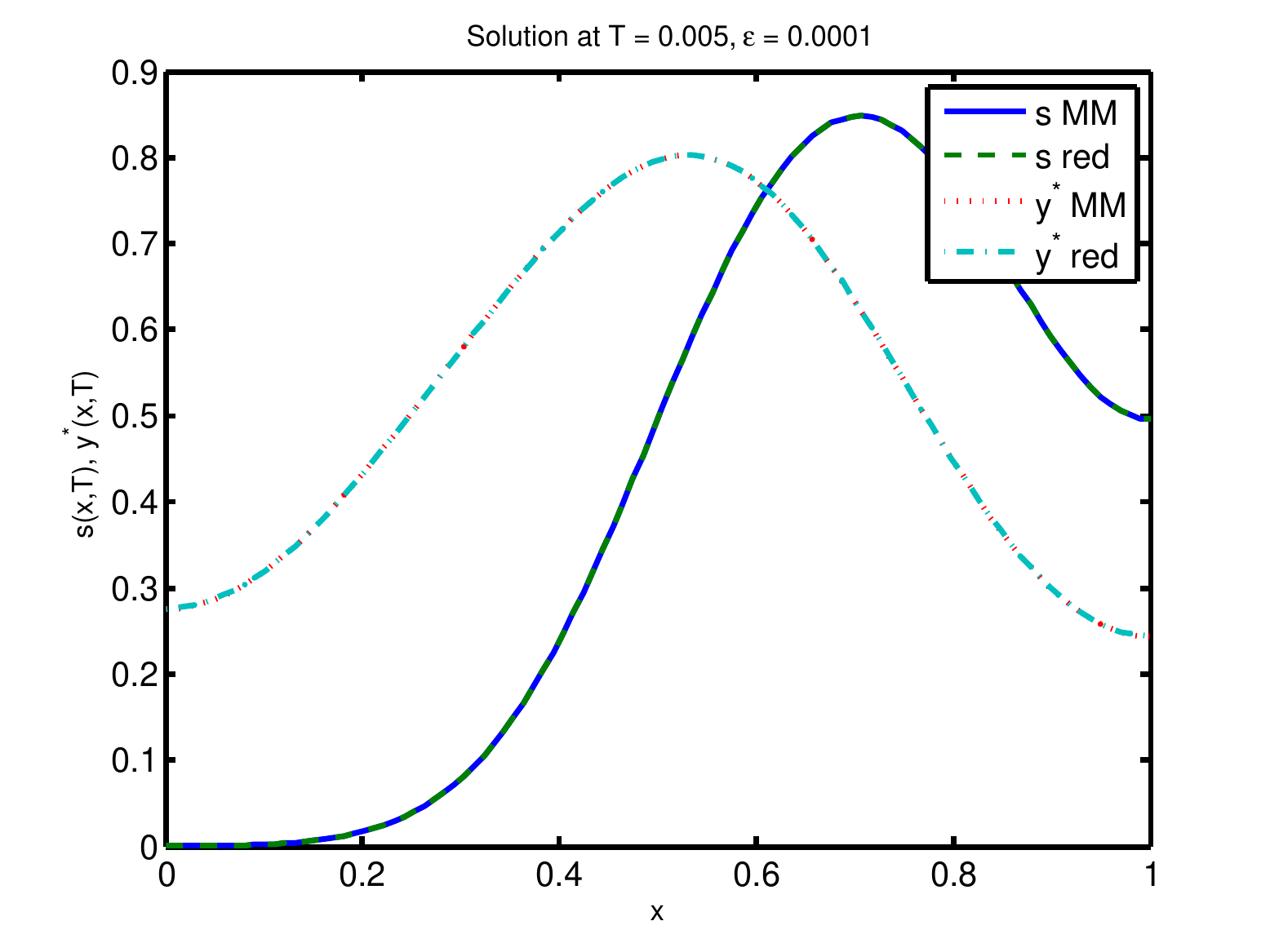}
\caption{Solutions $s$ and $y^*$ at time $T=0.005$. Comparison between Michaelis-Menten and reduced system for $\varepsilon=0.0001$.}
\label{fig:sol00001}
\end{figure}

\begin{figure}
\centering\includegraphics[width=0.8\linewidth]{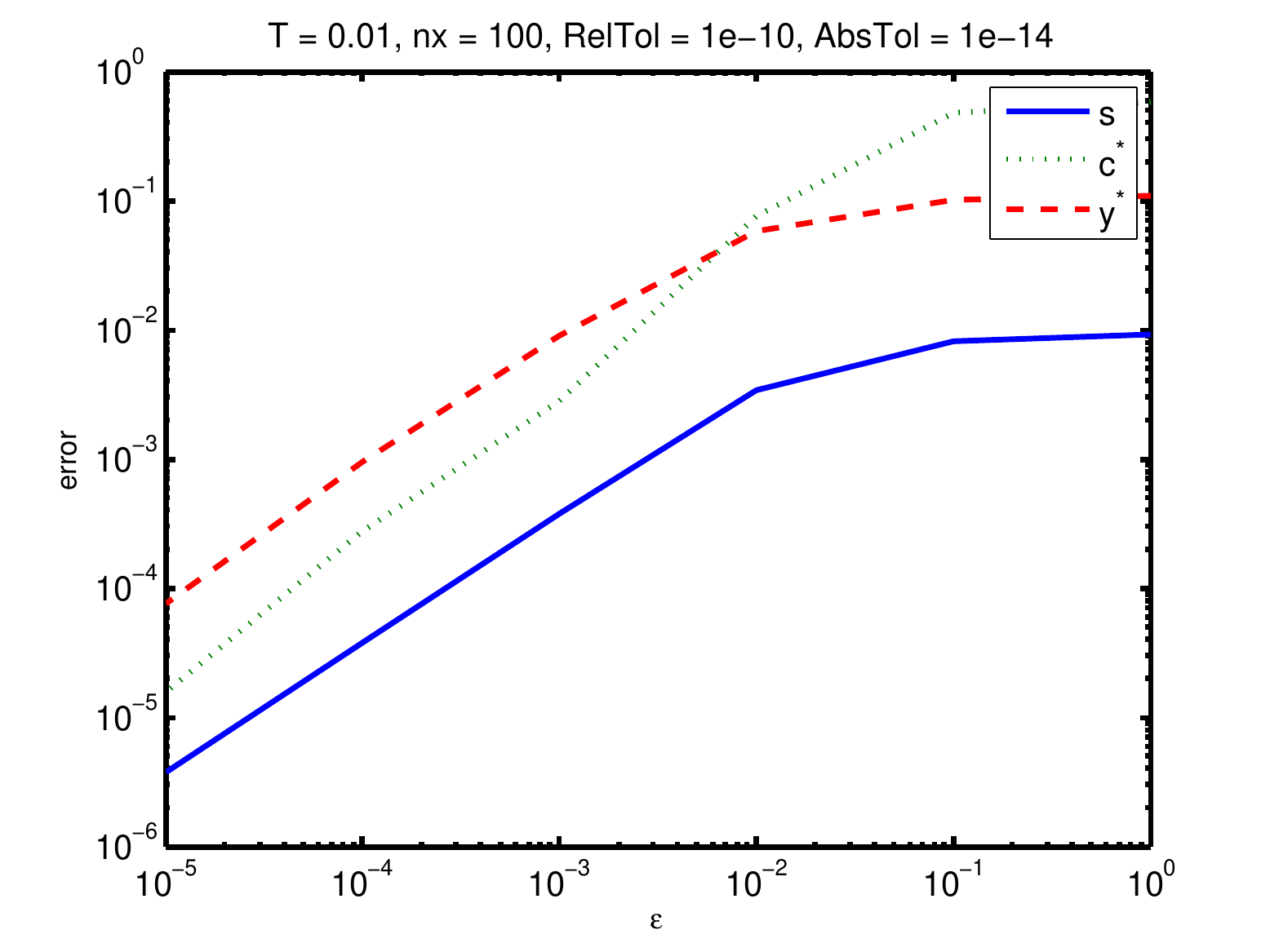}
\caption{Convergence of the full solution to the reduced solution as $\varepsilon\to 0$. Error measured in the $L^\infty$ norm.}
\label{fig:convergence}
\end{figure}

\begin{figure}
\centering\includegraphics[width=0.8\linewidth]{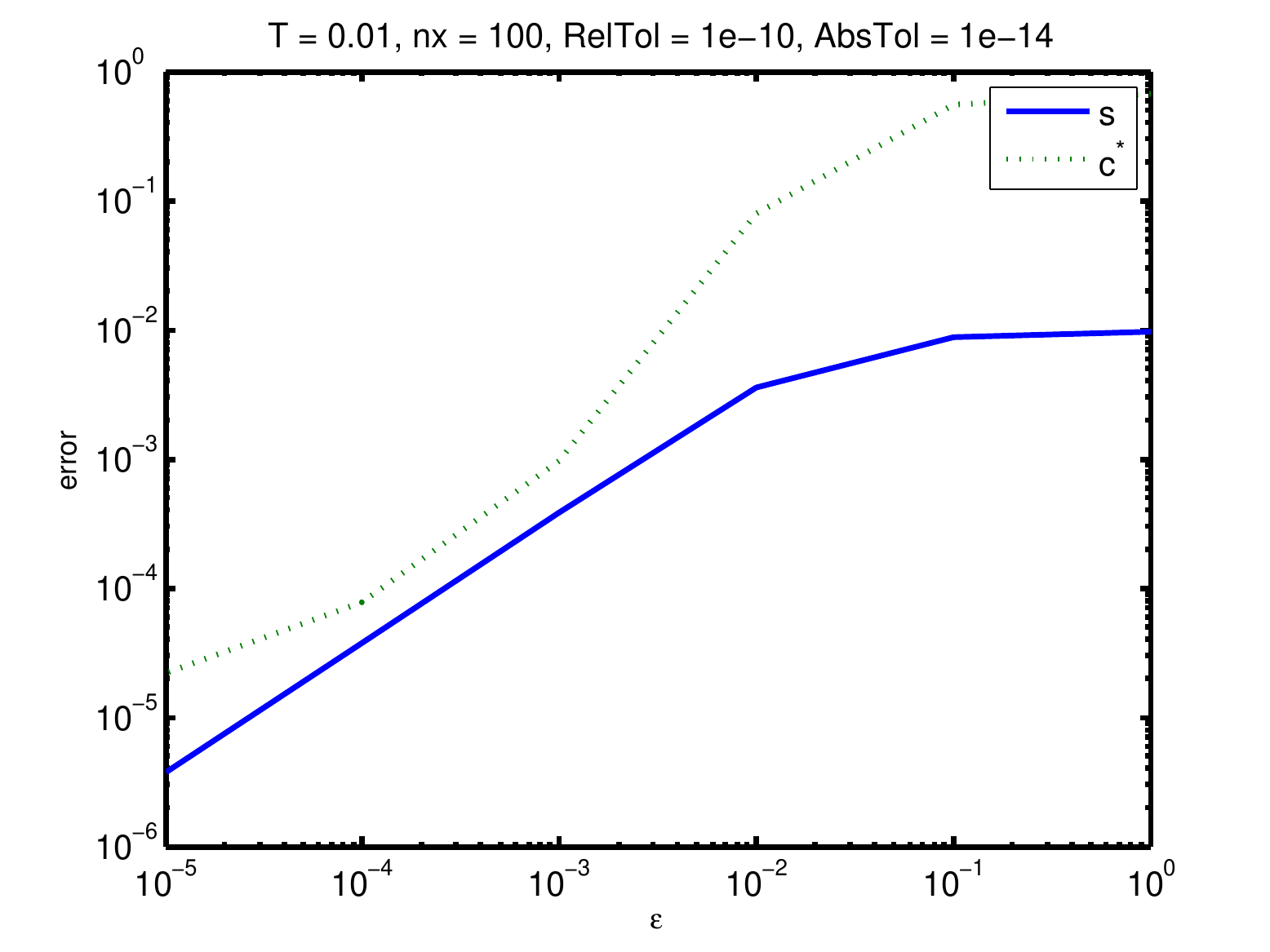}
\caption{Convergence of the full solution to the reduced solution for equal diffusion constants as $\varepsilon\to 0$. Error measured in the $L^\infty$ norm.}
\label{fig:convergenceglsc}
\end{figure}

\begin{figure}
\centering\includegraphics[width=0.8\linewidth]{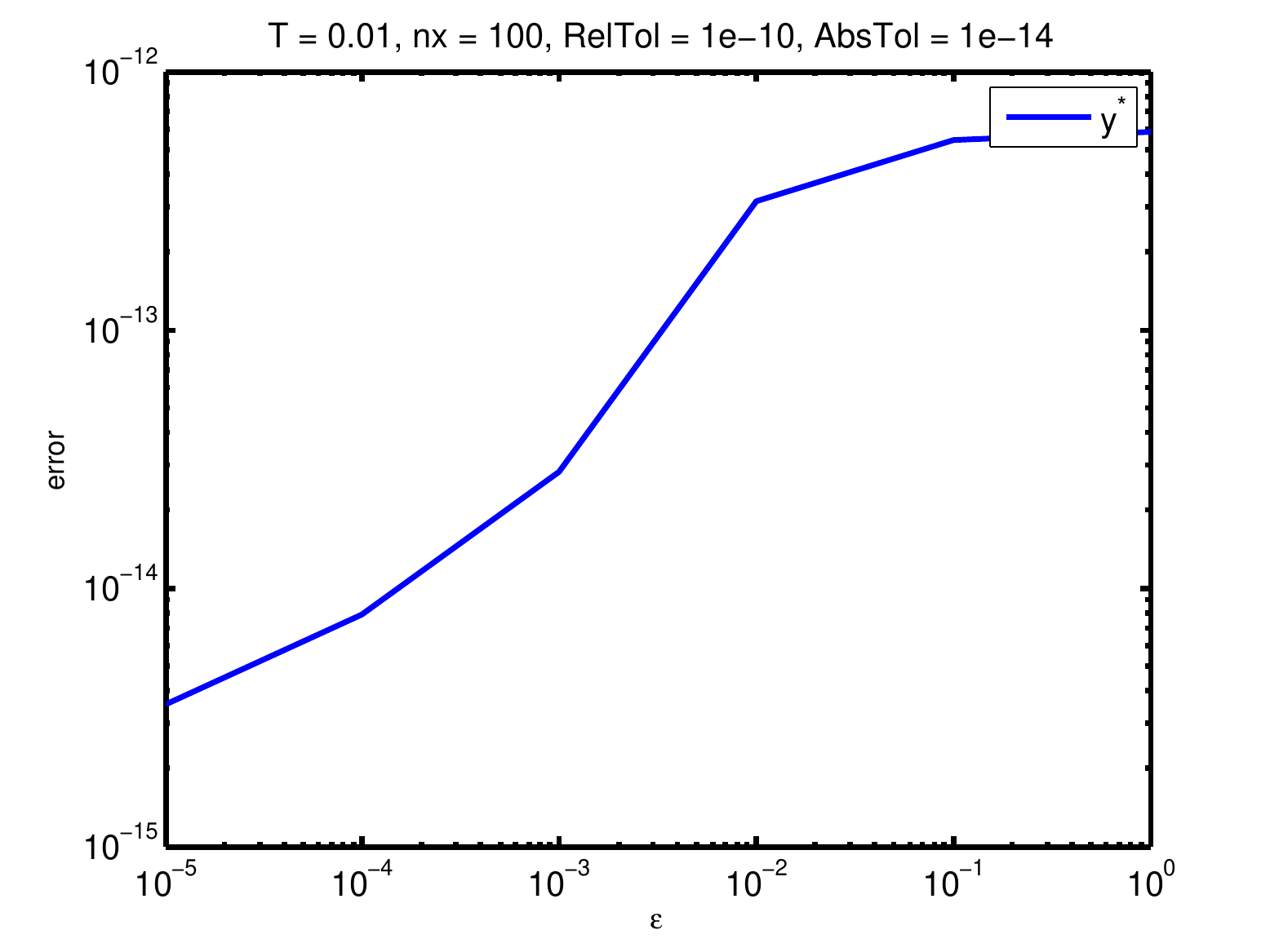}
\caption{Convergence of the full solution to the reduced solution for equal diffusion constants as $\varepsilon\to 0$. Error measured in the $L^\infty$ norm.}
\label{fig:convergencegly}
\end{figure}

%
%
%

\section{Concluding remarks}\label{conclrem}
 
\begin{itemize}
 \item As already noted, we do not discuss convergence results. But it is easy to see that the uniform bound for $y^*$ implies that $c$ converges uniformly to 0 as $\varepsilon\to0$. Moreover, up to taking a subsequence, $c^*:=\varepsilon^{-1}c$ and $y^*$ converge $\text{weakly}^*$ in $C^{0}$ and weakly in $L^{p}$ for all $1<p<\infty$. This may be a starting point for a convergence proof. 
 \item As already mentioned, this above reductions can be obtained only after a degenerate scaling of certain variables; then a Tikhonov-Fenichel reduction is applicable. (The corresponding scaling by Heineken et al. \cite{hta} in the ODE case  is convenient, but not necessary.) This may also be the underlying reason why the approach by Yannacopoulos et al. \cite{Yannacopoulos} was not directly applicable to the given setting. The scaled quantities $y^*$ and $c^*$ can be seen as first order approximations of $y$ and $s$ of the solution of \eqref{mm1diffy}--\eqref{mm3diffy} (with respect to the assumptions regarding slow diffusion and small total initial enzyme concentration) where the zero order terms are equal to zero. The effect of degenerate scalings in general  is investigated in a forthcoming paper \cite{lw2}.
 \item A reduction similar to the one above was already given in the dissertation \cite{Lax}, but it was based on writing the system in $(s,e,c,p)$ and scaling both $e=\varepsilon e^*$ and $c=\varepsilon c^*$. The reduced system is equivalent to the reduced system given here. We chose to change the variables to $(s,c,y,p)$ in order to emphasize the resemblance to the non-diffusive case which is otherwise lost. 
 \item It is also possible to only scale $y$ instead of both $c$ and $y$ (and still obtain that $c$ will be of order $\varepsilon$). But there are some disadvantages: The computation of the reduced system gets more involved as the results of \cite{laxgoeke} cannot be used directly. Moreover, we only get a zero order approximation to the slow manifold, given by $c=0$.
 \item Different QSS assumptions are also being discussed in the literature. Various choices of small rate constants can be found in \cite{laxgoeke,Lax}; for example the assumptions of slow product formation ($k_2=\varepsilon k_2^*$) and slow diffusion ($\delta_z=\varepsilon \delta_z^*\text{ for }z=s,e,c,p$) as well as only slow product formation are discussed.\\
 Moreover, the assumption of slow complex formation ($k_1=\varepsilon k_1^*$ and $k_{-2}=\varepsilon k_{-2}^*$) and slow diffusion can be discussed by employing the method developed in \cite{laxgoeke}. A reduced system is given by
  \begin{align*}
   \partial_{\tau} s&= \delta_s \Delta s-\frac{k_1k_2}{k_{-1}+k_2}se+\frac{k_{-1}k_{-2}}{k_{-1}+k_2} ep\\
   \partial_{\tau} e&= \delta_e\Delta e \\
   \partial_{\tau} p&= \delta_p\Delta p+\frac{k_1k_2}{k_{-1}+k_2}se-\frac{k_{-1}k_{-2}}{k_{-1}+k_2} ep
  \end{align*}
 on the slow manifold defined by $c=0$. This corresponds to the convergence results of Bothe and Pierre \cite{BothePierre1} and Bisi et al. \cite{bisi} for a related system which is defined by the reaction $A_1+A_2 \rightleftharpoons A_3 \rightleftharpoons A_4+A_5$. (Note that the latter reaction is easier to analyze, due to the structure of the conservation laws; see Elia\v{s} \cite{elias}). In all cases, the numerical results are in good agreement with the reduction.
 \item By analogous methods one can derive a reduction given the assumption of small total initial enzyme concentration, but with fast diffusive terms. Scaling again $y=\varepsilon y^*$ and $c=\varepsilon c^*$ and using results of \cite{Lax} one obtains the classical reduction: the fast diffusion yields a homogenization of the concentrations, enzyme and complex are in QSS and the reduced dynamics of the substrate are described by \eqref{red} (again, the reduction is in good agreement with numerical results). We omit details here.
\end{itemize}

\section{Acknowledgement}
The second-named author was supported by the DFG Research Training Group ``Experimental and Constructive Algebra'' (GRK 1632). 

\appendix

\section{Appendix: Computations and proofs}
Here we collect, for the reader's convenience, some known results and facts, and we present the proofs of some of the main results in Section 3 in detail, sketching the remaining ones.
\subsection{Tikhonov-Fenichel reduction of ODEs}
We first recall some results from \cite{gw2}. Consider a polynomial or rational system
\begin{equation}\label{startsystem}
\dot x= h(x,\varepsilon)=h^{(0)}(x)+\varepsilon h^{(1)}(x)+\cdots,\quad x\in\mathbb R^m,\quad \varepsilon>0,
\end{equation}
and in addition assume that there exists $x_0$ in the zero set $\mathcal V(h^{(0)})$ such that $\rank Dh^{(0)}(x)=r<m$ for all $x\in \mathbb R^m$ near $x_0$. We denote by $\V$ the irreducible component of $\mathcal V(h^{(0)})$ which contains $x_0$. By the implicit function theorem, there is a (Zariski-open) neighborhood $U$ of $x_0$ such that $\V\cap U$ is a $(m-r)$-dimensional submanifold. 
\begin{proposition}\label{agreduce} (See \cite{gw2}, Theorem 1.) Assume furthermore that
  \[
   \mathbb R^m=\ker Dh^{(0)}(x) \oplus \im Dh^{(0)}(x)
  \]
for all $x\in \V\cap U$, and that there exists $\nu>0$ such that all nonzero eigenvalues of $ Dh^{(0)}(x)$, $x\in \V$, have real part $\leq -\nu$.
 Then the following hold.
\begin{enumerate}[(a)]
\item There exist rational maps  \[P\colon \mathbb R^m\to \mathbb R^{m\times r}\quad \text{and} \quad \mu\colon\mathbb R^m\to \mathbb R^r\]  which are regular in $x_0$, with $\rank P(x_0)=\rank D\mu(x_0)=r$, such that the identity
    \[
     h^{(0)}(x)=P(x)\mu(x)
    \]
of rational functions holds. Moreover, the zero set $Y$ of $\mu$ satisfies $Y\cap \widetilde U=\V\cap \widetilde U$ in some Zariski-open neighborhood $\widetilde U$ of $x_0$.
\item The system
    \begin{equation}\label{Grenzsystem}
     x^\prime=q(x):=Q(x)\cdot h^{(1)}(x)
    \end{equation}
  with \[Q(x):=I_m-P(x)(D\mu(x)P(x))^{-1}D\mu(x),\] 
(in slow time $\tau=\varepsilon t$) is defined in $x_0$, and the manifold $Y\cap\widetilde U$ is an invariant set of \eqref{Grenzsystem}. 
\item There exists $T>0$ and a neighborhood $U^*\subset U$ of $Y$ such that, for any $\tau_0$ with $0<\tau_0<T$, solutions of 
\[
x^\prime = \varepsilon^{-1}h(x,\varepsilon)
\]
starting in $U^*$ converge uniformly on $[\tau_0,T]$ to solutions of the reduced system \eqref{Grenzsystem} on $Y$ as $\varepsilon\to 0$.
\end{enumerate}
\end{proposition}

We call \eqref{Grenzsystem} the Tikhonov-Fenichel reduction of \eqref{startsystem}. In order to apply this reduction, one also needs to know the appropriate initial value on $Y$. This was basically settled by Fenichel \cite{fenichel} Theorem 9.1 and was discussed in detail for the given particular setting in \cite{gw2} (see also the references given there). We briefly summarize: By \cite{gw2} Proposition 2, the system $\dot x=h^{(0)}(x)$ admits $m-r$ first integrals in a neighborhood of $x_0$. Moreover, the intersection of a common level set of the first integrals with $\V(h^{(0)})$ consists (locally) of a single point. Thus, to project the initial values of system \eqref{startsystem} to \eqref{Grenzsystem}, choose the corresponding intersection point.\\ 
In general it will not be possible to determine the first integrals explicitly (see \cite{gw2} Remark 6 for details on an approximation by Taylor series), but for Michaelis-Menten in this particular setting the first integrals are easily determined and the projected initial values (for the relevant variables $s$, $y^*$ and $p$) are identical to the original ones.

\subsection{Discretization}
Here we briefly review the spatial discretization procedure, and some properties of the discretized system; see \cite{laxgoeke} and \cite{Lax}.
For the sake of simplicity we assume that $\Omega=(0,L)$ is a real interval. (When $\Omega=(0,L_1)\times(0,L_2)\times\cdots\times(0,L_n)\subseteq\mathbb R^n$ holds, a similar discretization with obvious adjustments can be carried out. Note that the special choice of $\Omega$ is relevant only for the derivation of the reduced PDE system.)\\
Let $L=N\rho$, where $\rho$ is the mesh size. We subdivide $\Omega$ in compartments 
$
 \Omega_{\alpha}=\left((\alpha-1)\rho,\alpha\rho\right),
$
with $1\leq \alpha\leq N$, and we identify $\alpha$ with the compartment $\Omega_{\alpha}$. Define $z_{\alpha}$ as the concentration of species $Z$ at the center $x_\alpha$ of compartment $\alpha$ (with $Z=S,E,C,P$), and let
\[
\widehat z:=(z_\alpha)_{1\leq \alpha\leq N}\in \mathbb R^N.
\]
We choose a central difference discretization of the Laplacian, i.e.
  \begin{align*}
   \D_{\alpha}\widehat z=\frac{z_{\alpha-1}-2z_{\alpha}+z_{\alpha+1}}{\rho^2},\quad 1\leq \alpha\leq N.
  \end{align*}
To incorporate the Neumann boundary conditions, we set $z_{0}=z_{1}$ and $z_{N+1}=z_{N}$.\\ 
The discretization of \eqref{mm1diff}--\eqref{mm4diff} is given by
  \begin{align}
   \dot s_\alpha&= \delta_s {\cal D}_\alpha\widehat s-k_1s_\alpha e_\alpha+k_{-1} c_\alpha, &1\leq \alpha\leq N \label{mm1disc}\\
   \dot e_\alpha&= \delta_e {\cal D}_\alpha\widehat e-k_1s_\alpha e_\alpha+(k_{-1} +k_2)c_\alpha-k_{-2}e_\alpha p_\alpha, &1\leq \alpha\leq N  \label{mm2disc} \\
    \dot c_\alpha&= \delta_c {\cal D}_\alpha\widehat c+k_1s_\alpha e_\alpha-(k_{-1} +k_2)c_\alpha+k_{-2}e_\alpha p_\alpha, &1\leq \alpha\leq N  \label{mm3disc} \\
   \dot p_\alpha&= \delta_p {\cal D}_\alpha\widehat p +k_2c_\alpha-k_{-2}e_\alpha p_\alpha, &1\leq \alpha\leq N  \label{mm4disc}
  \end{align}
with initial values
\begin{align*}
&s_{\alpha}(0)=s_{\alpha, 0}:=s_0(x_{\alpha}),\quad e_{\alpha}(0)=e_{\alpha, 0}:=e_0(x_{\alpha}),\\ &c_{\alpha}(0)=c_{\alpha, 0}:=c_0(x_{\alpha}),\quad p_{\alpha}(0)=p_{\alpha, 0}:=p_0(x_{\alpha}).
\end{align*}
We will also make use of the total enzyme concentrations
\[
\widehat y:=\widehat e+\widehat c
\]
and the discretized diffusion matrix $\cal D$, which in dimension one has the form
\[
 \D:=\left(\D_{\alpha}\right)_{1\leq \alpha\leq N}=\frac1{\rho^2}\begin{pmatrix}
          -1 & 1 \\
    1 & -2 & 1\\
      & \ddots & \ddots & \ddots\\
    & & \ddots & \ddots & \ddots\\
     & & &	1 & -2 & 1\\
     & &   &       & 1 & -1
     \end{pmatrix}.
\]
(In spatial dimension $n>1$, the discretized diffusion matrix $\D$ for $(0,L_1)\times\cdots\times(0,L_n)$ is of different form but in any case it is a so-called $W$-matrix, i.e. the sum of all rows is equal to zero and all off-diagonal elements are nonnegative.)\\
Our fundamental assumptions for the discretized system correspond to those in Subsection \ref{mainresultirrev}; they are as follows:
\begin{enumerate}[(i)]
\item Diffusion is slow, thus one may scale
\begin{equation}\label{diffscale}
\delta_s=\varepsilon\delta_s^*,\quad \delta_c=\varepsilon\delta_c^*,\quad,\delta_e=\varepsilon\delta_e^*, \quad \delta_p=\varepsilon\delta_p^*
\end{equation}
with a small parameter $\varepsilon >0$.
\item The initial concentrations of enzyme and complex are small of order $\varepsilon$ in every compartment; hence there is a constant $E>0$ such that $e_{\alpha, 0}\leq E\cdot\varepsilon$ and $c_{\alpha, 0}\leq E\cdot\varepsilon$ for all $\alpha$.
\end{enumerate}
Requirement (ii) is less restrictive than the corresponding one for the PDE system, since the conditions refer only to $t=0$. Actually, for the ODE system after discretization the property for $t>0$ follows automatically, and we will prove this and some other basic properties for the discretized system next.

\begin{lemma}\label{elemlem}\begin{enumerate}[(a)]
\item For every $x\in\mathbb R^N$ one has $\sum_\alpha {\cal D}_\alpha(x)=0$.
\item Solutions with nonnegative initial values are nonnegative for all $t\geq 0$.
\item One has 
\[
\sum_\alpha(s_\alpha+e_\alpha+2c_\alpha+p_\alpha)=\sum_\alpha(s_{\alpha,0}+e_{\alpha,0}+2c_{\alpha,0}+p_{\alpha,0})
\]
for every $t\geq 0$; in particular every component of the solution is bounded.
\item There is a constant $C>0$ such that $y_\alpha=e_\alpha+c_\alpha\leq C\cdot \varepsilon$ for all $t\geq 0$, $1\leq \alpha\leq N$.
\end{enumerate}
\end{lemma}

\begin{proof} Part (a) reflects the property of the discretized diffusion matrix $\cal D$ that the sum of its rows equals zero. Part (b) is a consequence of the fact that off-diagonal elements of $\cal D$ are nonnegative, hence the rate of change for every variable $z_\alpha$ is nonnegative whenever $z_\alpha =0$. Part (c) follows from (a) and \eqref{mm1disc}  -- \eqref{mm4disc}, and part (d) follows from 
\[
\sum_\alpha(e_\alpha+c_\alpha)=\sum_\alpha(e_{\alpha,0}+c_{\alpha,0})\leq 2NE\cdot\varepsilon.
\]
\end{proof}
In view of part (d) of the Lemma, the scaling
\begin{equation}\label{varscale}
y_\alpha=\varepsilon y_\alpha^* \text{ and } c_\alpha=\varepsilon c_\alpha^*,\quad 1\leq \alpha\leq N
\end{equation}
is consistent; i.e., all $y_\alpha^*$ and $c_\alpha^*$ remain bounded in $\varepsilon$ for $t\geq 0$ whenever assumption (ii) holds.


\subsection{Reduction of the discretized irreversible system}
For the irreversible system (thus $k_{-2}=0$) one has
  \begin{align}
   \dot s_\alpha&= \delta_s {\cal D}_\alpha\widehat s-k_1s_\alpha e_\alpha+k_{-1} c_\alpha, &1\leq \alpha\leq N \label{mm1irrdisc}\\
   \dot e_\alpha&= \delta_e {\cal D}_\alpha\widehat e-k_1s_\alpha e_\alpha+(k_{-1} +k_2)c_\alpha, &1\leq \alpha\leq N  \label{mm2irrdisc} \\
    \dot c_\alpha&= \delta_c {\cal D}_\alpha\widehat c+k_1s_\alpha e_\alpha-(k_{-1} +k_2)c_\alpha, &1\leq \alpha\leq N  \label{mm3irrdisc}
  \end{align}
since the equations for $\widehat p$ may be omitted. Rewriting the system with 
$\widehat y=\widehat e+\widehat c$, 
we get
  \begin{align}
   \dot s_\alpha&= \delta_s {\cal D}_\alpha\widehat s-k_1s_\alpha (y_\alpha-c_\alpha)+k_{-1} c_\alpha, &1\leq \alpha\leq N \label{mm1irrdiscn}\\
   \dot c_\alpha&= \delta_c {\cal D}_\alpha\widehat c+k_1s_\alpha (y_\alpha-c_\alpha)-(k_{-1} +k_2)c_\alpha, &1\leq \alpha\leq N  \label{mm2irrdiscn}\\
 \dot y_\alpha&= \delta_e {\cal D}_\alpha\widehat y + (\delta_c-\delta_e){\cal D}_\alpha\widehat c, &1\leq \alpha\leq N.  \label{mm3irrdiscn} 
  \end{align}
With the scalings \eqref{diffscale} and \eqref{varscale},
system \eqref{mm1irrdiscn}-- \eqref{mm3irrdiscn}  becomes
 \begin{align}
   \dot s_\alpha&= \varepsilon\delta_s^* {\cal D}_\alpha\widehat s+\varepsilon(k_1s_\alpha +k_{-1}) c^*_\alpha-\varepsilon k_1s_\alpha y_\alpha^*, &1\leq \alpha\leq N \label{mm1irrdisceps}\\
   \dot c^*_\alpha&= \varepsilon\delta_c^* {\cal D}_\alpha\widehat c^*-(k_1s_\alpha +k_{-1}+k_2) c^*_\alpha+ k_1s_\alpha y_\alpha^*, &1\leq \alpha\leq N  \label{mm2irrdisceps}\\
 \dot y^*_\alpha&= \varepsilon\delta_e^* {\cal D}_\alpha\widehat y^* +\varepsilon \delta{\cal D}_\alpha\widehat c^*, &1\leq \alpha\leq N.  \label{mm3irrdisceps} 
  \end{align}
where $\delta=\delta_c^*-\delta_e^*$ as in \eqref{delteq}. The main result in Section \ref{mainresultirrev} is a direct consequence of the following Proposition; note that the reduced ODE system is the spatial discretization of \eqref{mm1diffred1} -- \eqref{mm3diffred1} resp.  \eqref{mm1diffred} -- \eqref{mm3diffred}.

\begin{proposition}\label{irrevdiscred}
\begin{enumerate}[(a)]
\item The Tikhonov-Fenichel reduction of system  \eqref{mm1irrdisceps}--\eqref{mm3irrdisceps} is given by
  \begin{align}
   s_\alpha^\prime&= \delta_s^* {\cal D}_\alpha\widehat s-\frac{k_1k_2y_\alpha^*s_\alpha}{k_1s_\alpha+k_{-1}+k_2}, &1\leq \alpha\leq N \label{mm1discred} \\
   {y_\alpha^*}^\prime&= \delta_e^*{\cal D}_\alpha\widehat  y^*+\delta {\cal D}_\alpha \left(\frac{k_1y_\beta^*s_\beta}{k_1s_\beta+k_{-1}+k_2}\right)_{1\leq \beta\leq N},  &1 \leq\alpha\leq N \label{mm3discred}
  \end{align}
on the asymptotic slow manifold determined by \[c^*_\alpha=\frac{k_1s_\alpha y_\alpha^*}{k_1s_\alpha +k_{-1}+k_2},\ 1\leq \alpha\leq N.\]
\item In the special case that the diffusion constants $\delta_e^*$ for enzyme and $\delta_c^*$ for complex are equal (or, more generally, whenever their difference is ${\cal O}(\varepsilon)$), we get the reduction
  \begin{align}
   s_\alpha^\prime&= \delta_s^* {\cal D}_\alpha\widehat s-\frac{k_1k_2y_\alpha^*s_\alpha}{k_1s_\alpha+k_{-1}+k_2}, &1\leq \alpha\leq N \label{mm1discred1} \\
   {y_\alpha^*}^\prime&= \delta_e^*{\cal D}_\alpha\widehat y^*,  &1 \leq\alpha\leq N. \label{mm3discred1}
  \end{align}
\item The corresponding initial values of the reduced system on the asymptotic slow manifold may be taken as
\[
\widetilde s_{\alpha,0}=s_{\alpha,0},\quad \widetilde y_{\alpha,0}^*=y_{\alpha,0}^*;\quad 1\leq \alpha\leq N.
\]
\end{enumerate}
\end{proposition}

\begin{proof} We first show part (a). In the terminology of Proposition \ref{agreduce} we have for \eqref{mm1irrdisceps}--\eqref{mm3irrdisceps}:
\[
h^{(0)}=P\cdot \mu
\] 
with
\[
\mu=\big(-(k_1s_\alpha +k_{-1}+k_2) c^*_\alpha+ k_1s_\alpha y_\alpha^*\big)_{1\leq \alpha\leq N}  
\]
and
\[
 P=\begin{pmatrix}0\\ I_N\\ 0\end{pmatrix}.
\]
In particular, we have 
\begin{equation}\label{hilfeslowman}
 h^{(0)}=0\iff c^*_\alpha=\frac{k_1s_\alpha y_\alpha^*}{k_1s_\alpha +k_{-1}+k_2},\ 1\leq \alpha\leq N.
\end{equation}
Moreover
\[
h^{(1)}=\begin{pmatrix}\left(\delta_s^* {\cal D}_\alpha\widehat s+(k_1s_\alpha +k_{-1}+k_2) c^*- k_1s_\alpha y_\alpha^*\right)_{1\leq \alpha\leq N}\\
\delta_c^* {\cal D}\widehat c^*\\
\delta_e^* {\cal D}\widehat y^* + \delta{\cal D}\widehat c^*
\end{pmatrix}.
\]
Following the procedure in Proposition \ref{agreduce} we obtain
\[
D\mu =\begin{pmatrix}M_1& M_2& M_3\end{pmatrix}
\]
with
  \begin{align*}
   M_1&={\rm diag}\left(-k_1c_1^*+k_1y_1^*,\ldots,-k_1c_N^*+k_1y_N^*\right) \\
   M_2&={\rm diag}\left(-k_1s_1+k_{-1}+k_2,\ldots,-k_1s_N+k_{-1}+k_2\right) \\
   M_3&={\rm diag}\left(k_1s_1,\ldots,k_1s_N\right).
  \end{align*}
Thus we have
\[
 D\mu P=M_2.
\]

Since all eigenvalues of $M_2$ are negative, the eigenvalue condition from Proposition \ref{agreduce} is satisfied (see e.g. \cite{gw2}, Remark 4). Furthermore
\[
Q=I_{3N}-\begin{pmatrix}0\\ I_N \\ 0\end{pmatrix}M_2^{-1}\begin{pmatrix}M_1& M_2& M_3\end{pmatrix}
=\begin{pmatrix} I_N&0&0\\
                             -M_2^{-1}M_1&0&-M_2^{-1}M_1\\
                         0&0&I_N\end{pmatrix}
\]
from which (a) follows by a straightforward computation, using \eqref{hilfeslowman}. \\
As for part (b), repeating the above procedure with 
\[
\widetilde h^{(1)}=\begin{pmatrix}\left(\delta_s^* {\cal D}_\alpha\widehat s+(k_1s_\alpha +k_{-1}+k_2) c^*- k_1s_\alpha y_\alpha^*\right)_{1\leq \alpha\leq N}\\
\delta_c^* {\cal D}\widehat c^*\\
\delta_e^* {\cal D}\widehat y^*
\end{pmatrix}
\]
yields the asserted result.\\
To prove part (c) we notice that the fast system 
   \begin{align}
   \dot s_\alpha&= 0, &1\leq \alpha\leq N \\
   \dot c^*_\alpha&= -(k_1s_\alpha +k_{-1}+k_2) c^*_\alpha+ k_1s_\alpha y_\alpha^*, &1\leq \alpha\leq N  \\
 \dot y^*_\alpha&= 0, &1\leq \alpha\leq N  
  \end{align}
possesses the first integrals
  \begin{align*}
   &\Psi_{s,\alpha}(\widehat s,\widehat c^*,\widehat y^*)=s_\alpha,\quad 1\leq \alpha\leq N\\
   &\Psi_{y^*,\alpha}(\widehat s,\widehat c^*,\widehat y^*)=y_\alpha^*,\quad 1\leq \alpha\leq N.
  \end{align*}
Thus, we get $\widetilde s_{\alpha,0}=s_{\alpha,0}$ and $\widetilde y_{\alpha,0}^*=y_{\alpha,0}^*$ for $1\leq \alpha\leq N$. For the sake of completeness we note that the appropriate initial value for $c^*_{\alpha}$ is given by 
  \[
   \widetilde c^*_{\alpha,0}=\frac{k_1s_{\alpha,0} y_{\alpha,0}^*}{k_1s_{\alpha,0} +k_{-1}+k_2},\ 1\leq \alpha\leq N.
  \]
\end{proof}

\subsection{The discretized reversible system}
Here we sketch the argument leading to the reduction in Subsection \ref{mainresultrev}. This is parallel to the irreversible case, hence we will present fewer details.
Using the scalings \eqref{diffscale} and \eqref{varscale} and rewriting the system in terms of $\widehat c^*$ and $\widehat y^*$,  \eqref{mm1disc}--\eqref{mm4disc} becomes
 \begin{align*}
   \dot s_\alpha&= \varepsilon\delta_s^* {\cal D}_\alpha\widehat s+\varepsilon(k_1s_\alpha +k_{-1}) c^*_\alpha-\varepsilon k_1s_\alpha y_\alpha^*, &1\leq \alpha\leq N \\
   \dot c_\alpha^*&= \varepsilon\delta_c^* {\cal D}_\alpha\widehat c^*-(k_1s_\alpha +k_{-1}+k_2+k_{-2}p_\alpha) c_\alpha^*+ (k_1s_\alpha+k_{-2}p_\alpha) y_\alpha^*, &1\leq \alpha\leq N  \\
 \dot y^*_\alpha&= \varepsilon\delta_e^* {\cal D}_\alpha\widehat y^* +\varepsilon \delta{\cal D}_\alpha\widehat c^*, &1\leq \alpha\leq N  \\
 \dot p_\alpha&= \varepsilon\delta_p^* {\cal D}_\alpha\widehat p + \varepsilon(k_2+k_{-2}p_\alpha) c_\alpha^*-\varepsilon k_{-2}p_\alpha y_\alpha^*, &1\leq \alpha\leq N.  
  \end{align*}
The computation of the reduced system proceeds as for the irreversible reaction; the only difference lies in the choice of
\[
\mu=\big(-(k_1s_\alpha +k_{-1}+k_2+k_{-2}p_\alpha) c^*_\alpha+ (k_1s_\alpha+k_{-2}p_\alpha) y_\alpha^*\big)_{1\leq \alpha\leq N}  
\]
and
\[
 P=\begin{pmatrix}0\\ I_N\\ 0 \\ 0\end{pmatrix}.
\]
One obtains:
\begin{proposition}\label{revdiscred}
For $\delta_c^*-\delta_e^*=\O(1)$, the reduced system is given by 
  \begin{align*}
   s_\alpha^\prime&= \delta_s^* {\cal D}_\alpha\widehat s-\frac{(k_1k_2s_\alpha-k_{-1}k_{-2}p_{\alpha})y_\alpha^*}{k_1s_\alpha+k_{-1}+k_2+k_{-2}p_\alpha}, &1\leq \alpha\leq N  \\
   {y_\alpha^*}^\prime&= \delta_e^*{\cal D}_\alpha\widehat  y^*+\delta {\cal D}_\alpha \left(\frac{(k_1s_\beta+k_{-2}p_{\beta})y_\beta^*}{k_1s_\beta+k_{-1}+k_2+k_{-2}p_\alpha}\right)_{1\leq \beta\leq N},  &1 \leq\alpha\leq N \\
   p_\alpha^\prime&= \delta_p^* {\cal D}_\alpha\widehat p+\frac{(k_1k_2s_\alpha-k_{-1}k_{-2}p_{\alpha})y_\alpha^*}{k_1s_\alpha+k_{-1}+k_2+k_{-2}p_\alpha}, &1\leq \alpha\leq N 
  \end{align*}
on the asymptotic slow manifold determined by \[c^*_\alpha=\frac{(k_1s_\alpha+k_{-2}p_\alpha) y_\alpha^*}{k_1s_\alpha +k_{-1}+k_2+k_{-2}p_\alpha},\ 1\leq \alpha\leq N.\]
For $\delta_c^*-\delta_e^*=\O(\varepsilon)$, formally setting $\delta=0$ in the above system yields the correct reduction.\\
Appropriate initial values on the asymptotic slow manifold are given by
\[\tilde s_{\alpha,0}=s_{\alpha,0},\quad \tilde y^*_{\alpha,0}=y^*_{\alpha,0}\quad \text{and}\quad \tilde p_{\alpha,0}=p_{\alpha,0}.\] 
\end{proposition}
\bibliography{bibliography}
\bibliographystyle{abbrv}

\end{document}